\documentclass{article} 

\usepackage[all]{xy}
\usepackage[english]{babel}
\usepackage{amsmath}
\usepackage{amsthm}
\usepackage{amssymb}

\newtheorem{thm}{Theorem}

\newtheorem{prop}{Proposition}

\theoremstyle{definition}

\theoremstyle{remark}
\newtheorem{rem}{Remark}

\theoremstyle{remark}

\def\cM{\mathcal M}

\def\cH{\mathcal H}

\def\bR{\mathbb R} 

\def\bN{\mathbb N}

\def\bC{\mathbb C}

\def\bH{\mathbb H}

\def\Re{\operatorname{Re}}

\def\Im{\operatorname{Im}}

\def\fai{\varphi}

\def\lz{\langle}
\def\pz{\rangle}
\def\pa{\partial}
\def\nad{\overline}

\def\sl{\mathfrak{sl}}
\def\so{\mathfrak{so}}

\begin{document}

\title{Canonical bases for $\sl(2,\bC)$-modules of spherical monogenics in dimension 3}

\author{Roman L\' avi\v cka\thanks{The financial support from the grant GA 201/08/0397 is gratefully acknowledged.
This work is also a part of the research plan MSM 0021620839, which is financed by the Ministry of Education of the Czech Republic.}
\footnote{Mathematical Institute, Charles University, Sokolovsk\'a 83, 186 75 Praha 8, Czech Republic;
email: \texttt{lavicka@karlin.mff.cuni.cz}}
}

\date{}

\maketitle

\begin{abstract}
Spaces of homogeneous spherical monogenics in dimension 3 can be considered naturally as  $\sl(2,\bC)$-modules. As finite-dimensional irreducible $\sl(2,\bC)$-modules, they have canonical bases which are, by construction, orthogonal.
In this note, we show that these orthogonal bases form the Appell system and coincide with those constructed recently by S. Bock and K. G\"urlebeck in \cite{BG}. Moreover, we obtain simple expressions of elements of these bases in terms of the Legendre polynomials.
\end{abstract}

\medskip\noindent
{\small

}

\section{Introduction}

The main aim of this paper is to present an easy way to construct explicitly orthogonal bases for spaces of homogeneous spherical monogenics in dimension 3. 
Such bases were recently obtained by I. Ca\c c\~ao in \cite{cac} and  by S. Bock and  K. G\"urlebeck in \cite{BG}.
In \cite{cac}, orthogonal bases are constructed from systems of spherical monogenics which are obtained by applying the adjoint Cauchy-Riemann operator to elements of the standard bases of spherical harmonics.
In \cite{BG}, this idea is used for producing another orthogonal bases of spherical monogenics forming, in addition, the Appell system.  
In \cite{BGLS}, 
it is observed that these bases forming the Appell system can be seen as the so-called Gelfand-Tsetlin bases.
Moreover, in \cite{BGLS}, it is shown that the Gelfand-Tsetlin bases could be obtained in quite a~different way using the Cauchy-Kovalevskaya method
and a~characterization of the bases is given there. 
In \cite[Theorem 2.2.3, p. 315]{DSS}, the Cauchy-Kovalevskaya method was already explained. But this method is not used in \cite{DSS} for a~construction of orthogonal bases of spherical monogenics although the construction is obvious not only in dimension 3 but in an arbitrary dimension as well.       
Actually, in \cite[pp. 254-264]{DSS} and \cite{som,van}, another constructions even in all dimensions are given.
By the way, the Cauchy-Kovalevskaya method is applicable in other settings, see \cite{BDLS1,BDLS2} and \cite{DLS4}.
Finally, let us remark that
Appell systems of monogenic polynomials were discussed before
by H. R. Malonek et al. in \cite{CM06,CM07,FCM,FM}.  
Similar questions were also studied for the Riesz system, see \cite{GM06,GM07,GM09,del07,zei}.

For an account of Clifford analysis, we refer to \cite{DSS}. Now we introduce some notations.
Let $(e_1,\ldots,e_m)$ be the standard basis of the Euclidean space $\bR^m$
and let $\bC_m$ be the complex Clifford algebra generated by the vectors $e_1,\ldots,e_m$ such that
$e_j^2=-1$ for $j=1,\ldots,m.$
Recall that the Spin group $Spin(m)$ is defined as the set of products of an even number of unit vectors of $\bR^m$ endowed with the Clifford multiplication.
The Lie algebra $\so(m)$ of the group $Spin(m)$ can be realized as the space of bivectors of Clifford algebra $\bC_m,$
that is, $$\so(m)=\lz \{e_{ij}: 1\leq i<j\leq
m\}\pz.$$ Here $e_{ij}=e_ie_j$ and $\lz M\pz$ stands for the span of
a~set $M.$

Denote by $\cH_k(\bR^3)$ the space of complex valued harmonic polynomials $P$ in $\bR^3$ which are $k$-homogeneous. 
Then the space $\cH_k(\bR^3)$ of spherical harmonics is an irreducible module under the $h$-action,
defined by
$$
[h(s)(P)](x) = P(s^{-1}xs),\ s\in Spin(m)\text{\ \ and\ \ }x=(x_1,x_2,x_3)\in\bR^3.
$$
Moreover, let $S$ be a~basic spinor representation of the group $Spin(3).$ Then denote by $\cM_k(\bR^3,S)$ the set of $S$-valued $k$-homogeneous polynomials $P$ in $\bR^3$ which satisfy the equation $\pa P=0$
where the Dirac operator $\pa$ is given by
$$\pa=e_1\frac{\pa\ }{\pa x_1}+e_2\frac{\pa\ }{\pa x_2}+e_3\frac{\pa\ }{\pa x_3}.$$
It is well-known that the space $\cM_k(\bR^3,S)$ of spherical monogenics is an irreducible module under the  $L$-action,
defined by
$$
[L(s)(P)](x) = s\,P(s^{-1}xs),\ s\in Spin(m)\text{\ \ and\ \ }x=(x_1,x_2,x_3)\in\bR^3.
$$

Both spaces $\cH_k(\bR^3)$ and $\cM_k(\bR^3,S)$ can be seen naturally as irreducible finite-dimensional $\sl(2,\bC)$-modules. 
As finite-dimensional irreducible $\sl(2,\bC)$-modules, they have canonical bases which are, by construction, orthogonal.

In this paper, we study properties of canonical bases of spaces $\cM_k(\bR^3,S).$ 
In Theorem \ref{x-}, we describe their close relation to canonical bases of spherical harmonics, we show that they form the Appell system and we give recurrence formulas for their elements. By the way, in \cite{boc,BG09} analogous recurrence formulas generate easily elements of the orthogonal bases described in \cite{BG}.  
Moreover, we express elements of the canonical bases in terms of classical special functions (see Theorem	 \ref{sphGT3}). As in \cite{BGLS}, we can adapt these results easily to quaternion valued spherical monogenics. 
It turns out that these bases coincide with those constructed recently by S. Bock and K. G\"urlebeck in \cite{BG}.
In Theorem \ref{sphappell}, we obtain simple expressions of elements of these bases in terms of the Legendre polynomials.
Let us remark that in \cite{GM10a,GM10b} homogeneous solutions of the Riesz system in dimension 3 forming orthogonal bases are expressed as finite sums of products of the Legendre and Chebyshev polynomials.

\section{Spherical harmonics in dimension 3}

In this section, we recall the construction of canonical bases for finite-dimensional irreducible $\sl(2,\bC)$-modules and, as an example, we describe well-known bases of   
spherical harmonics in dimension 3 by means of classical special functions.

Obviously, the action of $\so(3)$ on the space $\cH_k(\bR^3)$ is given by
$$h_{ij}=dh(e_{ij}/2)=x_j\frac{\pa}{\pa x_i}-x_i\frac{\pa}{\pa x_j}\ \ \ (i\not=j).$$
Moreover, it is easily seen that 
$$[h_{12},h_{23}]=h_{31},\ \ \ [h_{23},h_{31}]=h_{12}\text{\ \ \ and\ \ \ }[h_{31},h_{12}]=h_{23}$$
where $[L,K]=LK-KL.$
We can naturally  identify the Lie algebra $\sl(2,\bC)$ with the complexification of $\so(3).$
Indeed, the operators
$$H=-ih_{12},\ \ X^{+}=h_{31}+ih_{23}\text{\ \ and\ \ }X^{-}=-h_{31}+ih_{23}$$ satisfy the standard $\sl(2,\bC)$-relations: 
$$[X^{+},X^{-}]=2H\text{\ \ \ and\ \ \ }[H,X^{\pm}]=\pm X^{\pm}.$$
Putting $z=x_1+ix_2$ and $\nad z=x_1-ix_2,$ we have that
\begin{equation}\label{Xpm}
X^{+}=-2x_3\frac{\pa}{\pa z}+\nad z\frac{\pa}{\pa x_3}\text{\ \ and\ \ }
X^{-}=2x_3\frac{\pa}{\pa\nad z}- z\frac{\pa}{\pa x_3}.
\end{equation}
Furthermore, it is well-known that, as an $\sl(2,\bC)$-module, $\cH_k(\bR^3)$ is irreducible and has the highest weight $k.$ In each finite-dimensional irreducible $\sl(2,\bC)$-module there exists always a~canonical basis consisting of weight vectors, see \cite[p. 116]{BtD}.

\begin{prop}\label{basis_sl2}
Let $V_l$ be an irreducible $\sl(2,\bC)$-module with the highest weight $l.$ 
Then 

\medskip\noindent
(i) There is a~primitive element $f_0$ of $V_l,$ that is, there is a~non-zero element $f_0$ of $V_l$ such that 
$$Hf_0=l f_0\text{\ \ and\ \ }X^+f_0=0.$$ 
 
\medskip\noindent
(ii) A~basis of $V_l$ is formed by the elements
$$f_j=(X^{-})^jf_0,\ \ j=0,\ldots, 2l.$$  
In addition, 
for each $j=0,\ldots,2l,$ 
the element $f_j$ is a~weight vector with the weight $l-j,$ that is, 
$f_j$ is a~non-zero element of $V_l$ such that $$Hf_j=(l-j)f_j.$$ 
Moreover, $X^-f_{2l}=0$ and each weight vector $f_j$ is uniquely determined up to a~non-zero multiple.

\medskip\noindent
(iii) The basis $\{f_0,\ldots,f_{2l}\}$ is orthogonal with respect to any inner product $(\cdot,\cdot)$ on $V_l$ which is invariant, that is, 
for each $L\in\sl(2,\bC)$ and each $f,g\in V_l,$ we have that $$(Lf,Lg)=(f,g).$$   
\end{prop}

By Proposition \ref{basis_sl2}, to construct the canonical basis of the module $\cH_k(\bR^3)$ it is sufficient to find its primitive.  

\begin{prop}\label{GT3harm}
The irreducible $\sl(2,\bC)$-module $\cH_k(\bR^3)$ has a~basis consisting of the polynomials
$$f^k_0=\frac{1}{k!\; 2^k  }\; \nad z^k\text{\ \ \ and\ \ \ }f_j^k=(X^{-})^jf^k_0,\ \ 0<j\leq 2k.$$
In addition, for each $j=0,\ldots,2k,$ 
the polynomial $f_j^k$ is a~weight vector with the weight $k-j,$ that is, 
$Hf_j^k=(k-j)f_j^k.$ 
\end{prop}

\begin{proof}
It is easy to see that $f^k_0$ is a~primitive of $\cH_k(\bR^3).$ 
\end{proof}

Following \cite{BtD}, we identify the functions $f^k_j$ with classical special functions. To do this we use spherical co-ordinates
\begin{equation}\label{sph}
x_1=r\sin\theta\sin\fai,\ \
x_2=r\sin\theta\cos\fai,\ \ 
x_3=r\cos\theta  
\end{equation}
with $0\leq r,$ $-\pi\leq\fai\leq \pi$ and $0\leq\theta\leq\pi.$
Let us remark that, in  spherical co-ordinates (\ref{sph}), the operators $H,$ $X^+$ and $X^-$ have the form
\begin{equation}\label{HXsph}
\begin{array}{l}
H=-i\frac{\pa\ }{\pa\fai}, \medskip\\{}
X^+=e^{i\fai}\left(i\frac{\pa\ }{\pa\theta}-\cot\theta\frac{\pa\ }{\pa\fai}\right), \medskip\\{}
X^-=e^{-i\fai}\left(i\frac{\pa\ }{\pa\theta}+\cot\theta\frac{\pa\ }{\pa\fai}\right).
\end{array}
\end{equation}
In \cite[pp. 120-121]{BtD} (with the variables $x_1$ and $x_2$ interchanged), the next result is shown.  

\begin{prop}\label{sphGT3harm}
Let $\{f^k_0,\ldots,f^k_{2k}\}$ be the basis of $\cH_k(\bR^3)$ defined in Proposition \ref{GT3harm}.
Using  spherical co-ordinates (\ref{sph}), we have then that, for each $j=0,\ldots,2k,$
$$f^k_j(r,\theta,\fai)=i^{k-j} r^k e^{i(k-j)\fai}P^{j-k}_k(\cos \theta)\text{\ \ where} $$
$$P^l_k(s)=\frac{1}{k!\; 2^k  }\; (1-s^2)^{l/2}\frac{d^{l+k}}{ds^{l+k}}\;(s^2-1)^k,\ s\in\bR.$$
Here $P^0_k$ is the $k$-th Legendre polynomial and $P^l_k$ are its associated Legendre functions.
\end{prop}


\section{Spherical monogenics in dimension 3}

In this section, we study properties of canonical bases of $\sl(2,\bC)$-modules of spherical monogenics and, in particular, we express elements of these bases by means of classical special functions. We begin with spinor valued spherical monogenics. 

\paragraph{Spinor valued polynomials}

In what follows,  $S$ stands for a~(unique up to equivalence) basic spinor representation of $Spin(3)$ and $\so(3)=\lz e_{12}, e_{23}, e_{31}\pz.$
As an $\so(2)$-module, the module $S$ is reducible and decomposes into two inequivalent irreducible submodules 
$$S^{\pm}=\{u\in S: -ie_{12}\;u=\pm u\}$$
provided that $\so(2)=\lz e_{12}\pz.$ 
Moreover, the spaces $S^{\pm}$ are both one-dimensional. 
Let $S^{\pm}=\lz v^{\pm}\pz.$ Then
each $s\in S$ is of the form $s=s^{+}v^{+}+s^{-}v^{-}$ for some complex numbers $s^{\pm}.$ We write $s=(s^{+},s^{-}).$ 

Furthermore, the action of $\so(3)$ on the space $\cM_k(\bR^3,S)$ is given by
$$L_{ij}=dL(e_{ij}/2)=\frac{e_{ij}}{2}+h_{ij}\text{\ \ with\ \ }h_{ij}=x_j\frac{\pa}{\pa x_i}-x_i\frac{\pa}{\pa x_j}\ \ \ (i\not=j).$$
It is easily seen that 
$$[L_{12},L_{23}]=L_{31},\ \ \ [L_{23},L_{31}]=L_{12}\text{\ \ \ and\ \ \ }[L_{31},L_{12}]=L_{23}.$$
Moreover, the operators
$$\tilde H=-iL_{12},\ \ \tilde X^{+}=L_{31}+iL_{23}\text{\ \ and\ \ }\tilde X^{-}=-L_{31}+iL_{23}$$ generate the Lie algebra $\sl(2,\bC).$ Indeed, we have that 
$$[\tilde X^{+},\tilde X^{-}]=2\tilde H\text{\ \ \ and\ \ \ }[\tilde H,\tilde X^{\pm}]=\pm \tilde X^{\pm}.$$
Put again $z=x_1+ix_2$ and $\nad z=x_1-ix_2.$ Then it is easy to see that 
$$\tilde X^{\pm}=X^{\pm}+\omega^{\pm}
\text{\ \ where\ \ } 
\omega^+=\frac 12(e_{31}+ie_{23}),\ \ \
\omega^-=\frac 12(-e_{31}+ie_{23})$$
and $X^{\pm}$ are defined as in (\ref{Xpm}). 
Furthermore, as an $\sl(2,\bC)$-module, $\cM_k(\bR^3,S)$ is irreducible and has the highest weight $k+\frac 12.$ We can construct again a~canonical basis of this module using Proposition \ref{basis_sl2}.

\begin{prop}\label{GT3}
The irreducible $\sl(2,\bC)$-module $\cM_k(\bR^3,S)$ has a~basis consisting of the polynomials
$$F^k_0=\frac{1}{k!\; 2^k  }\; \nad z^k v^+\text{\ \ \ and\ \ \ }F_j^k=(\tilde X^{-})^jF^k_0,\ \ 0<j\leq 2k+1.$$
In addition, for each $j=0,\ldots,2k+1,$ 
the polynomial $F_j^k$ is a~weight vector with the weight $k+\frac 12-j,$ that is, 
$\tilde H F_j^k=(k+\frac 12-j)F_j^k.$ 
\end{prop}

\begin{proof}
Obviously, the polynomial $F^k_0$ is a~primitive of $\cM_k(\bR^3,S).$ 
\end{proof}

By Proposition \ref{basis_sl2}, the basis of $\cM_k(\bR^3,S)$ constructed in Proposition \ref{GT3} is orthogonal with 
respect to any invariant inner product on $\cM_k(\bR^3,S).$ As is well-known, the Fischer inner product and the standard $L^2$-inner product on the unit ball of $\bR^3$ are examples of invariant inner products on $\cM_k(\bR^3,S),$ see \cite[pp. 206 and 209]{DSS}.
In the next theorem, we show further properties of the constructed bases.
Statement (a) of Theorem \ref{x-} shows the close relation of the canonical bases of spherical harmonics to those of spherical monogenics. Moreover, by statement (b), the polynomials $F^k_j$ form the so-called Appell system, that is, they satisfy the property (\ref{appellsystem}) below. Finally, statement (c) of Theorem \ref{x-} contains the recurrence formula for elements $F^k_j$ of the constructed bases.  

\begin{thm}\label{x-}
(a) 
We have that $$(\tilde X^-)^j=(X^-)^j+j(X^-)^{j-1}\omega^-,\ \ j\in\bN.$$
In particular, for $j=0,\ldots,2k+1,$ we get that
$F^k_j=f^k_jv^+ + jf^k_{j-1}\omega^-v^+.$
Here $f^k_{-1}=f^k_{2k+1}=0$ and $\{f^k_0,\ldots,f^k_{2k}\}$ is the basis of $\cH_k(\bR^3)$ as in Proposition \ref{GT3harm}.

\medskip\noindent
(b) Moreover, it holds that 
$$[\frac{\pa}{\pa x_3}, (\tilde X^-)^j]=2j(\tilde X^-)^{j-1}\frac{\pa}{\pa \nad z},\ \ j\in\bN.$$ 
In particular, for each $k\in\bN,$ 
\begin{equation}\label{appellsystem}
\frac{\pa F^k_j}{\pa x_3}
=\left\{
\begin{array}{ll}
j\; F^{k-1}_{j-1},&\ \ \ j=1,\ldots,2k;\medskip\\{} 
0,&\ \ \ j=0,2k+1. 
\end{array}
\right. 
\end{equation}

\noindent
(c) Finally, we have that $$[x_3, (\tilde X^-)^j]=j(\tilde X^-)^{j-1}z,\ \ j\in\bN.$$ 
In particular, 
for each $k\in\bN_0$ and $j=0,1,\ldots, 2k+1,$
$$F^{k+1}_{j+1}=x_3F^k_j-jzF^k_{j-1}+\omega^- F^{k+1}_j\text{\ \ where\ \ }F^k_{-1}=0.$$
\end{thm}

\begin{proof}
The statements (a) and (b) follow, by induction, from the following facts:
$$(\omega^-)^2=0,\ \ [X^-,\omega^- ]=0,\ \ [\frac{\pa}{\pa x_3},X^-]=2\frac{\pa}{\pa \nad z}\text{\ \ and\ \ }[\frac{\pa}{\pa \nad z},X^-]=0.$$
We show statement (c). We have that $[x_3,X^-]=z$ and $[z,X^-]=0$ and hence, by induction, 
we get easily $$[x_3, (\tilde X^-)^j]=j(\tilde X^-)^{j-1}z.$$
In particular, for $j= 1,\ldots,2k,$ we have that
$$x_3F^k_j-(\tilde X^-)^j(x_3F^k_0)=jzF^k_{j-1},$$
which finishes the proof together with the obvious relation
$$F^{k+1}_1=x_3F^k_0+\omega^- F^{k+1}_0.\mbox{\qedhere}$$
\end{proof}

\begin{rem}\label{Fpm} (a) 
We can realize the space $S$ in the Clifford algebra $\bC_4.$ Indeed, we can put
$$v^{+}=\frac 14(1-ie_{12})(1-ie_{34})\text{\ \ \ and\ \ \ }v^{-}=\frac 14(e_1+ie_2)(e_3+ie_4).$$
We denote this realization of the space $S$ by $S^+_4.$
In particular, we have that $\omega^- v^{+}=v^{-}$ and $\omega^- v^{-}=0.$

\medskip\noindent
(b) There is another realization $S^-_4$ of the space $S$ inside $\bC_4$ if we put 
$$v^{+}=\frac 14(1-ie_{12})(e_3+ie_4)\text{\ \ \ and\ \ \ }v^{-}=\frac 14(e_1+ie_2)(1-ie_{34}).$$
In this case, we have that $\omega^- v^{+}=-v^{-}$ and $\omega^- v^{-}=0.$
Let us remark that although, as $\so(3)$-modules, $S^+_4$ and  $S^-_4$ are of course equivalent to each other they are different as $\so(4)$-modules. See \cite[pp. 114-118]{DSS} for details.

\medskip\noindent
(c) Let $\{F^{k,\pm}_0,\ldots,F^{k,\pm}_{2k+1}\}$ be the basis of $\cM_k(\bR^3,S^{\pm}_4)$ defined in Proposition \ref{GT3}. 
By statement ($a$) of Theorem \ref{x-}, it is easy to see that, for $j=0,\ldots,2k+1,$ we get 
$$F^{k,\pm}_j=(f^k_j,\;\pm jf^k_{j-1}).$$ 
Here $\{f^k_0,\ldots,f^k_{2k}\}$ is the basis of $\cH_k(\bR^3)$ defined in Proposition \ref{GT3harm}.
\end{rem}

Using the observation (c) of Remark \ref{Fpm} and Proposition \ref{sphGT3harm}, we can easily express the functions $F^{k,\pm}_j$ in terms of classical special functions.

\begin{thm}\label{sphGT3}
Let $\{F^{k,\pm}_0,\ldots,F^{k,\pm}_{2k+1}\}$ be the basis of $\cM_k(\bR^3,S^{\pm}_4)$ defined in Proposition \ref{GT3}.
Using spherical co-ordinates (\ref{sph}), we then have that
$$F^{k,\pm}_j(r,\theta, \fai)=i^{k-j}r^ke^{i(k-j)\fai}\;(P^{j-k}_k(\cos \theta),\;\pm ije^{i\fai} P^{j-k-1}_k(\cos \theta))$$
for each $j=0,\ldots,2k+1.$ Here $P^{k+1}_k=0=P^{-k-1}_k.$
\end{thm}

Now we are going to deal with quaternion valued spherical monogenics.

\paragraph{Quaternion valued polynomials}

In what follows, $\bH$ stands for the skew field of real quaternions
$q$ with the imaginary units $i_1,$ $i_2$ and $i_3,$ that is,
$$i_1^2=i_2^2=i_3^2=i_1i_2i_3=-1\text{\ \ and\ \ }q=q_0+q_1i_1+q_2i_2+q_3i_3, (q_0,q_1,q_2,q_3)\in\bR^4.$$
For a~quaternion $q,$ put $\nad q=q_0-q_1i_1-q_2i_2-q_3i_3.$
We realize $\bH$ as the subalgebra of complex
$2\times 2$ matrices of the form
\begin{equation}\label{Hmatrix}
q=\begin{pmatrix}
q_0+iq_3 & -q_2+iq_1\\
q_2+iq_1 & q_0-iq_3
\end{pmatrix}.
\end{equation}
In particular, we have that
$$i_1=\begin{pmatrix}
0 & i\\
i & 0
\end{pmatrix},\ \ \
i_2=\begin{pmatrix}
0 & -1\\
1 & 0
\end{pmatrix}\text{\ \ \ and\ \ \ }
i_3=\begin{pmatrix}
i & 0\\
0 & -i
\end{pmatrix}.$$

Furthermore, we identify $\so(3)$ with $\lz i_1,i_2,i_3\pz$ as follows: $e_{12}\simeq i_3,$ $e_{23}\simeq i_1$ and $e_{31}\simeq i_2.$
Then we can realize the basic spinor representation $S$ of $\so(3)$ as the space $\bC^2$ of
column vectors
$$s=\begin{pmatrix}
q_0+iq_3 \\
q_2+iq_1
\end{pmatrix}
.$$
Here the action of $\so(3)$ on $S$ is given by the matrix multiplication from
the left.

Now we are interested in quaternion valued polynomials $g=g(y)$ in the
variable $y=(y_0,y_1,y_2)$ of $\bR^3.$
Let us denote by $\cM_k(\bR^3,\bH)$ the space of $\bH$-valued $k$-homogeneous
polynomials $g$ satisfying the Cauchy-Riemann
equation $Dg=0$ with
$$D=\frac{\pa\ }{\pa y_0}+i_1\frac{\pa\ }{\pa y_1}+i_2\frac{\pa\ }{\pa y_2}.$$
It is easy to see that both columns of an element $g\in\cM_k(\bR^3,\bH)$ belong to the space $\tilde\cM_k(\bR^3,S)$ 
of $S$-valued solutions $h$ of the equation $Dh=0$ which are $k$-homogeneous. 

Moreover, we can consider naturally $\cM_k(\bR^3,\bH)$ as a~right $\bH$-linear
Hilbert space with the $\bH$-valued inner product
$$(Q,R)_{\bH}=\int_{B_3}\nad QR\;dV$$
where $B_3$ is the unit ball and $dV$ is the Lebesgue measure in $\bR^3.$ 
In \cite{BG}, orthogonal bases of spaces $\cM_k(\bR^3,\bH)$ forming, in addition, the Appell system are constructed.
In \cite{BGLS}, 
the following characterization of these bases is given.     

\begin{prop}\label{appell}
For each $k\in\bN_0,$ there exists an orthogonal basis
\begin{equation}\label{Hbases}
\{g^k_j|\ j=0,\ldots,k\}
\end{equation}
of the right $\bH$-linear
Hilbert space $\cM_k(\bR^3,\bH)$ such that:

\medskip\noindent
(i) Let $j=0,\ldots,k$ and let $h^k_j$ and $h^k_{2k+1-j}$ be the first and the second column of the (matrix valued) polynomial $g^k_j,$
respectively. Then we have that 
$$Hh_j^k=(k+\frac 12-j)h^k_j\text{\ \ \ and\ \ \ }Hh^k_{2k+1-j}=-(k+\frac 12-j)h^k_{2k+1-j}\text{\ \ \ with}$$
$$H=-i(\frac{i_3}{2}+y_2\frac{\pa\ }{\pa y_1}-y_1\frac{\pa\ }{\pa y_2}).$$

\noindent
(ii) We have that $$\frac{\pa g^k_j}{\pa y_0} =\left\{
\begin{array}{ll}
k g^{k-1}_{j-1},&\ \ \ j=1,\ldots,k;\medskip\\{} 0,&\ \ \
j=0.
\end{array}
\right. $$

\noindent
(iii) For each $k\in\bN_0,$ we have that $g^k_0=(y_1-i_3y_2)^k.$

\medskip\noindent
Moreover, the polynomials $g^k_j$ are determined uniquely by the conditions (i), (ii) and (iii). Finally, for each $k\in\bN_0,$ the $S$-valued polynomials $$h^k_0,\ h^k_1,\ldots, h^k_{2k+1}$$ form the canonical basis of the $\sl(2,\bC)$-module $\tilde\cM_k(\bR^3,S).$
\end{prop}

In \cite{BG} and \cite{BGLS}, quite explicit formulas for the polynomials $g^k_j$ are given in the cartesian coordinates $y_0,y_1,y_2.$ We now construct these polynomials in yet another way using Theorem \ref{sphGT3}.
Indeed, in Theorem \ref{sphappell} below, we express the polynomials $g^k_j$ in spherical co-ordinates
\begin{equation}\label{sph2}
y_0=r\cos\theta,\ \ 
y_1=r\sin\theta\cos\fai,\ \ 
y_2=r\sin\theta\sin\fai
\end{equation}
with $0\leq r,$ $-\pi\leq\fai\leq \pi$ and $0\leq\theta\leq\pi.$

\begin{thm}\label{sphappell}
Let the set $\{g^k_j|\ j=0,\ldots,k\}$ be the basis of $\cM_k(\bR^3,\bH)$ as in Proposition \ref{appell}. 
Using  spherical co-ordinates (\ref{sph2}), we have then that 
$$
g^k_{j}(r,\theta,\fai)=(k!/j!)(-2)^{k-j}r^k\;(g^k_{j,0}+g^k_{j,1}\;i_1+g^k_{j,2}\;i_2+g^k_{j,3}\;i_3)\text{\ \ where}
$$
$$
\begin{array}{ll}
g^k_{j,0}=P^{j-k}_k(\cos \theta)\cos (j-k)\fai,& 
g^k_{j,1}=-j P^{j-k-1}_k(\cos \theta) \cos (j-k-1)\fai,\medskip\\{} 
g^k_{j,2}=j P^{j-k-1}_k(\cos \theta)\sin (j-k-1)\fai,& 
g^k_{j,3}=P^{j-k}_k(\cos \theta)\sin (j-k)\fai.
\end{array}
$$ 
Here $P^0_k$ is the $k$-th Legendre polynomial and $P^l_k$ are its associated Legendre functions (see Proposition \ref{sphGT3harm} for the formulas of $P^l_k$).
\end{thm}

\begin{proof} 
(a) Let $k\in\bN_0$ and $j=0,\ldots,k.$ It is easy to see that  the $S$-valued polynomial
$$\hat h^k_j(y_0,y_1,y_2)=F^{k,-}_j(-y_2, y_1, y_0)$$
solves the equation $D\hat h^k_j=0.$
Here $F^{k,-}_j$ are as in Remark \ref{Fpm} (c).

\medskip\noindent
(b) We can find non-zero complex numbers
$c^k_j\in\bC$ such that the polynomials $h^k_j=c^k_j\hat h^k_j$ satisfy, in addition, that
$h^k_0=((y_1-iy_2)^k,0)$ and
$$
\frac{\pa h^k_j}{\pa y_0} =\left\{
\begin{array}{ll}
k h^{k-1}_{j-1},&\ \ \ j=1,\ldots,k;\medskip\\{}
0,&\ \ \ j=0.
\end{array}
\right.
$$
Indeed, by Theorem \ref{x-}, it is sufficient to put
$c^k_j=(2i)^{k-j} k!/j!.$

\medskip\noindent
(c) Using  spherical co-ordinates (\ref{sph2}), we obviously have that
$$h^k_j(r,\theta,\fai)=c^k_j\;F^{k,-}_j(r,\theta,-\fai)$$
where $F^{k,-}_j$ are as in Theorem \ref{sphGT3}.
In particular, putting $d^k_j=(k!/j!)(-2)^{k-j},$ we have that $h^k_j=(h^k_{j,0},h^k_{j,1})$ with
$$h^k_{j,0}=d^k_j r^k e^{i(j-k)\fai}P^{j-k}_k(\cos \theta)\text{\ \ and\ \ }
h^k_{j,1}= d^k_j r^k (-i)je^{i(j-k-1)\fai}P^{j-k-1}_k(\cos \theta).$$ 

\medskip\noindent
(d) Finally, we define an $\bH$-valued polynomial $g^k_j$ corresponding to the $S$-valued polynomial $h^k_j=(h^k_{j,0},h^k_{j,1})$ by
$$g^k_j=\Re h^k_{j,0}+i_1\Im h^k_{j,1}+i_2\Re h^k_{j,1}+i_3\Im h^k_{j,0}.$$
Here,
for a~complex number $z,$ we write $\Re z$ for its real part and $\Im z$ for its imaginary part.
Obviously, 
the polynomials $g^k_j$ satisfy the conditions (i), (ii) and (iii) of Proposition \ref{appell}, which easily completes the proof.
\end{proof}

\subsection*{Acknowledgments}
I am grateful to V. Sou\v cek for useful conversations.



\begin{thebibliography}{99}
\def\topsep{0pt}
\def\parsep{0pt plus 5pt minus 1pt}
\def\itemsep{-0.5ex} 
\small
%
\bibitem{boc} S. Bock, \"Uber funktionentheoretische Methoden in der r\"aumlichen Elastizit\"atstheorie
(German), Ph.D-thesis, University Weimar, 2010
(see {\tt http://e-pub.uni-weimar.de/frontdoor.php?source\_opus=1503}).
%
\bibitem{BG09}
S. Bock and K. G\"urlebeck, On an Orthonormal Basis of Solid Spherical Monogenics Recursively Generated by Anti-Holomorphic $\bar z$-Powers, 
In: Proc. of ICNAAM 2009 (T.E. Simos, G. Psihoyios,
and Ch. Tsitouras, eds.), AIP Conference Proceedings, vol. 1168, 2009, pp. 765-768.
%
\bibitem{BG} S. Bock and K. G\"urlebeck, On a~generalized Appell
system and monogenic power series, Math. Meth. Appl. Sci. 33 (2010) (4), 394-411.
%
\bibitem{BGLS} S. Bock, K. G\"urlebeck, R. L\' avi\v cka and V. Sou\v cek, The Gelfand-Tsetlin bases for spherical monogenics in dimension 3, preprint.
%
\bibitem{BDS} F. Brackx, R. Delanghe, F. Sommen, Clifford analysis, Pitman, London, 1982.
%
\bibitem{BDLS1} F.\ Brackx, H.\ De Schepper, R.\ L\'{a}vi\v{c}ka, V.\ Sou\v{c}ek,
The Cauchy-Kovalevskaya	Extension Theorem in Hermitean Clifford Analysis, preprint.
%
\bibitem{BDLS2} F.\ Brackx, H.\ De Schepper, R.\ L\'{a}vi\v{c}ka, V.\ Sou\v{c}ek,
Gelfand-Tsetlin Bases of Orthogonal Polynomials in Hermitean Clifford Analysis, preprint.
%
\bibitem{BtD} T. Br\"ocker and T. tom Dieck, Representations of compact Lie groups, Springer, New York, 1985.
%
\bibitem{cac} I. Ca\c c\~ao, Constructive approximation by monogenic polynomials, Ph.D-thesis, Univ. Aveiro, 2004.
%
\bibitem{CM06} I. Ca\c c\~ao and H. R. Malonek, Remarks on some properties of monogenic polynomials,
In: Proc. of ICNAAM 2006, (T.E. Simos, G. Psihoyios, and Ch. Tsitouras, eds.), Wiley-VCH, Weinheim, 2006,
pp. 596-599.
%
\bibitem{CM07} I. Ca\c c\~ao and H. R. Malonek, On a complete set of hypercomplex Appell polynomials,
In: Proc. of ICNAAM 2008, (T. E. Timos, G. Psihoyios, Ch. Tsitouras, Eds.),
AIP Conference Proceedings, vol. 1048, 2008, 647-650.
%
\bibitem{del07} R. Delanghe, On homogeneous polynomial solutions of the Riesz system and their harmonic potentials,
Complex Var. Elliptic Equ.  52  (2007),  no. 10-11, 1047--1061.
%
%
\bibitem{DLS4} R. Delanghe, R. L\'avi\v cka and V. Sou\v cek,
The Gelfand-Tsetlin bases for Hodge-de Rham systems in Euclidean spaces, preprint.
%
\bibitem{DSS} R. Delanghe, F. Sommen and V. Sou\v cek, Clifford algebra and spinor-valued functions, Kluwer Academic Publishers, Dordrecht, 1992.
%
\bibitem{FCM} M. I. Falc\~ao, J. F. Cruz, and H. R. Malonek, Remarks on the generation of monogenic
functions, In: Proc. of IKM 2006, ISSN 1611-4086
(K. G\"urlebeck and C. K\"onke, eds.), Bauhaus-University Weimar, 2006.
%
\bibitem{FM} M. I. Falc\~ao and H. R. Malonek, Generalized exponentials through Appell sets in $\bR^{n+1}$ and
Bessel functions, In: Proc. of ICNAAM 2007 (T.E. Simos, G. Psihoyios,
and Ch. Tsitouras, eds.), AIP Conference Proceedings, vol. 936, 2007, pp. 750-753.
%
\bibitem{GM} J. E. Gilbert and M. A. M. Murray, Clifford Algebras and Dirac Operators in Harmonic Analysis,
Cambridge University Press, Cambridge, 1991.
%
\bibitem{GM06}
K. G\"urlebeck and J. Morais,
On monogenic primitives of Fueter polynomials,
In: Proc. of ICNAAM 2006, (T.E. Simos, G. Psihoyios, and Ch. Tsitouras, eds.), Wiley-VCH, Weinheim, 2006,  pp. 600-605.	
%
\bibitem{GM07}
K. G\"urlebeck and J. Morais,
On the calculation of monogenic primitives,
Adv. appl. Clifford alg. 17 (2007) (3), 481-496. 
%
\bibitem{GM09}
K. G\"urlebeck and J. Morais,
Bohr Type Theorem for Monogenic Power Series,
Computational Methods and Function Theory 9 (2009) (2), 633-651.
%
\bibitem{GM10a}
K. G\"urlebeck and J. Morais, On the development of Bohr's phenomenon in the context of Quaternionic analysis and related problems,  
arXiv:1004.1188v1  [math.CV], 2010.
%
\bibitem{GM10b}
K. G\"urlebeck and J. Morais, Real-Part Estimates for Solutions of the Riesz System in $\bR^3,$ 
arXiv:1004.1191v1  [math.CV], 2010.
%
\bibitem{som} F. Sommen, Spingroups and spherical means III, Rend. Circ. Mat. Palermo
(2) Suppl. No 1 (1989), 295-323.
%
\bibitem{sud} A. Sudbery, Quaternionic analysis, Math. Proc. Cambridge Phil. Soc. 85 (1979), 199-225.
%
\bibitem{van} P. Van Lancker, Spherical Monogenics: An Algebraic Approach, Adv. appl. Clifford alg. 19 (2009), 467–496.
%
\bibitem{zei} P. Zeitlinger, Beitr\"age zur Clifford Analysis und deren Modifikation (German), Ph.D-thesis, University Erlangen, 2005.
%
\end{thebibliography}
\end{document}